\documentclass[12pt,reqno]{amsart}
\usepackage{amsfonts,amssymb,amsmath}
\usepackage{datetime}
\usepackage{xcolor}
\usepackage{enumerate}

\definecolor{fgreen}{RGB}{44,144, 14}

\renewenvironment{proof}{{\bfseries Proof.}}{\qed}

\numberwithin{equation}{section} 
\newtheorem{theorem}{Theorem}[section] 
\newtheorem{proposition}[theorem]{Proposition} 
\newtheorem{corollary}[theorem]{Corollary} 
\newtheorem{lemma}[theorem]{Lemma} 

\theoremstyle{definition}
\newtheorem{definition}[theorem]{Definition} 
\newtheorem{remark}[theorem]{Remark}

\usepackage[colorlinks=true,citecolor=cyan,urlcolor=blue,linkcolor=blue]{hyperref}
\everymath{\displaystyle}

\def\R{\mathbb R}

\def\C{\mathbb C}

\def\H{\mathbb H}

\def\ib{\mathbf {i}}
\def\jb{\mathbf {j}}
\def\kb{\mathbf {k}}

\def\N{\mathbb N}

\def\R{\mathbb R}

\newcommand{\GL}{\mathrm{GL}}

\def\P{\mathbb P}

\def\R{\mathbb {R}}
\def\C{\mathbb {C}}
\def\N{\mathbb {N}}
\def\H{\mathbb {H}}

\def\ib{\mathbf {i}}

\def\jb{\mathbf {j}}

\def\GL{\rm GL}

\newcommand{\defref}[1]{Definition~\ref{#1}}

\newcommand{\thmref}[1]{Theorem~\ref{#1}}
\newcommand{\lemref}[1]{Lemma~\ref{#1}}
\newcommand{\remref}[1]{Remark~\ref{#1}}
\newcommand{\propref}[1]{Proposition~\ref{#1}}
\newcommand{\corref}[1]{Corollary~\ref{#1}}

\makeatletter
\@namedef{subjclassname@2020}{\textup{2020} Mathematics Subject Classification}
\makeatother

\begin{document}

	\title[Classification and Decomposition  in  $\mathrm{PSL}(3,\mathbb{H})$ ]{Classification and Decomposition of Quaternionic Projective Transformations }
	\author[S.  Dutta, K. Gongopadhyay and T.  Lohan]{Sandipan Dutta, Krishnendu Gongopadhyay and 
		Tejbir Lohan}
	
	\address{Indian Institute of Science Education and Research (IISER) Mohali,
		Knowledge City,  Sector 81, S.A.S. Nagar 140306, Punjab, India}
	\email{sandipandutta98@gmail.com }
	
	\address{Indian Institute of Science Education and Research (IISER) Mohali,
		Knowledge City,  Sector 81, S.A.S. Nagar 140306, Punjab, India}
	\email{krishnendug@gmail.com, krishnendu@iisermohali.ac.in}

	\address{Indian Institute of Science Education and Research (IISER) Mohali,
		Knowledge City,  Sector 81, S.A.S. Nagar 140306, Punjab, India}
	\email{tejbirlohan70@gmail.com}

	
	
	\date{\today}
	\begin{abstract}
We consider the projective linear group $\mathrm{PSL}(3,\H)$.  We have investigated the reversibility problem in this group and use the reversibility to offer an algebraic characterization of the dynamical types of $\mathrm{PSL}(3,\H)$.   We further decompose elements of $\mathrm{SL}(3,\H)$ as products of simple elements, where an element $g$ in ${\rm SL}(3, \H)$ is called \emph{simple} if it is conjugate to an element of ${\rm SL}(3, \R)$.  We have also revisited real projective transformations and following Goldman's ideas,  have offered a complete classification for elements of 
 $\mathrm{SL}(3,\R)$. 
\end{abstract} 

	\maketitle 
	
\section{Introduction}\label{sec-intro}
In classical two-dimensional hyperbolic geometry,  the linear group ${\rm SL}(2, \R)$ acts as the orientation-preserving isometries and the action is given by the Möbius transformations. Based on the fixed-point dynamics, the isometries are classified into three mutually exclusive types: elliptic, parabolic, and hyperbolic. It is a remarkable phenomenon that this dynamical classification can be characterized algebraically by the traces of the matrices representing the isometries. The situation in dimension three is very similar where the isometries are represented by the group ${\rm SL}(2, \C)$;  see  \cite[Theorems 4.3.1 and 4.3.4]{be}. The isometries in four and five-dimensional hyperbolic geometries have similar representations over the real quaternions $\H$; see \cite{Wi}. In these dimensions, the isometry group may be identified with groups of $2 \times 2$ quaternionic matrices. The above characterization for isometries of the four-dimensional hyperbolic space has been generalized by Cao, Parker, and Wang in \cite{cpw}. Algebraic characterization for isometries of the five-dimensional hyperbolic space has been obtained by Gongopadhyay \cite{KG}, and Parker and Short \cite{PS} independently using different approaches.    For other attempts towards algebraic characterization of the dynamical types of quaternionic Möbius transformations; see \cite{Ca},  \cite{Fo}. The above classifications also have their counterparts in complex and quaternionic hyperbolic geometries; see  \cite{CG}, \cite{Goldman}, \cite{KG2}, \cite{GPP}. In particular,  we mention Goldman's classification for ${\rm SU}(2,1)$, which acts as the isometry group of the two-dimensional complex hyperbolic space. In \cite[Theorem 6.2.4]{Goldman},  Goldman considered the discriminant of the characteristic polynomial of an element of ${\rm SU}(2,1)$ and characterized the dynamical types of isometries using this polynomial. \\

In the three-dimensional (real) hyperbolic geometry, the group ${\rm SL}(2, \C)$ acts on the boundary of the hyperbolic space, which may be identified with $ \P_{\C}^1$, and this action extends to the entire three-dimensional ball. The action of ${\rm SL}(2, \C)$ on $ \P_{\C}^1$ is significant to complex dynamics to study the behaviour of iterated sequences of functions. Generalizing this idea, Seade and Verjovsky initiated an investigation of dynamical properties of ${\rm SL}(3, \C)$ action on $ \P_{\C}^2$; see \cite{SV01}, \cite{SV02}. It is possible to classify the elements of ${\rm SL}(3, \C)$ using their dynamics into elliptic, parabolic, and loxodromic subclasses. Essentially such a classification follows from the Iwasawa decomposition of an element in ${\rm SL}(3, \C)$. In \cite{Na}, Navarette formalized different types of elements in ${\rm SL}(3, \C)$ using their Jordan forms and dynamical action on $ \P_{\C}^2$. Navarette further characterized these subclasses using a polynomial coming from the discriminant of the characteristic polynomial, as in Goldman's theorem mentioned above. \\
	
Let ${\rm SL}(n,  \H)$ be the group of $n \times n$ matrices over the quaternions  with 
quaternionic determinant one.  In five-dimensional (real) hyperbolic geometry,   the boundary of the hyperbolic space may be identified with the one-dimensional quaternionic projective space $ \P_{\H}^1$. The group ${\rm SL}(2, \H)$ acts on $ \P_{\H}^1$ by the quaternionic Möbius transformations that identifies ${\rm PSL}(2, \H)$ with the group of conformal diffeomorphisms of $ \P_{\H}^1$.  In \cite{KG},  \cite{PS}, this action has been used to classify the isometries in terms of numerical invariants. \\

It is a natural question to ask for higher dimensional generalizations of such classification. The problem of finding conjugacy invariants for ${\rm SL}(3, \H)$ is also a problem of independent interest. In \cite{KL}, Kim and Luo tackled this problem. They used ideas similar to  Gongopadhyay \cite{KG} and  Foreman \cite{Fo} to obtain formulas for the coefficients of characteristic polynomials of complex representations of quaternionic matrices as conjugacy invariants. \\

Recently, the authors have initiated an investigation of dynamics of ${\rm SL}(3,  \H)$ action on $ \P_{\H}^2$; see \cite{dgt}. The elements of ${\rm SL}(3, \H)$ can be divided into three broad classes, viz. elliptic, parabolic, and loxodromic, by their dynamics, and into further subclasses by combining them with the Jordan forms; see  \cite{dgt}. It is a natural problem to ask for algebraic characterization of these types using conjugacy invariants. However, because of non-commutativity, the quaternionic linear algebra differs significantly from the complex linear algebra. Accordingly, finding conjugacy invariants is more subtle for quaternionic matrices. One may try to use the same method as in \cite{KG} to algebraically characterize the dynamical types of elements of ${\rm SL}(3, \H)$. However, this turns out to be complicated, especially when the eigenvalues belong to mutually disjoint classes. \\

We offer an approach to overcome this difficulty in this paper. The approach is to use the notion of  \textit{reversibility}.   Recall that an element $g$ in a group $G$ is called \emph{reversible} if $g$ is conjugate to $g^{-1}$.  This notion is closely related to \emph{strongly reversible} elements, which are products of at most two involutions,  i.e.,  $g$ is conjugate to $g^{-1}$ in $G$ by an element of order at most $2$. The reversible elements are also known as `real' or `reciprocal' elements in the literature; see \cite{ST}, \cite{Sa}. It has been a problem of broad interest to find the equivalence of these notions and to classify such elements. Despite many works, concrete classification of reversibility depends very much on the local structure of a group, and it has been known only for a few families of analytic groups. We refer to the monograph \cite{FS} for a survey of reversibility in geometry and dynamics.    In \cite{LFS},   Lavicka et al. classified the reversible and strongly reversible elements in the group  $ \mathrm{SL}(2,\mathbb{H})$ using their action on $ \P_{\H}^1$.   Understanding reversibility in the context of groups $ \mathrm{SL}(3,\mathbb{H})$ and $ \mathrm{PSL}(3,\mathbb{H}) := \mathrm{SL}(3,\H)/ \{\pm \mathrm{I}_3\}$  is a problem of independent interest.\\

In this paper, we have investigated reversibility in $ \mathrm{SL}(3,\mathbb{H})$, and offer algebraic characterizations of elements of ${\rm PSL}(3, \H) $ using reversibility of the elements. We consider the group $\mathrm{PSL}(3,\H)$ of projective automorphisms of $ \P_{\H}^2$.    Every element $[g]$ in ${\rm PSL}(3, \H)$ has two lifts $g$ and $-g$ in ${\rm SL}(3, \H)$. For this purpose, we first classify reversible and strongly reversible elements in ${\rm SL}(3, \H)$. However, that does not completely classify reversible elements in ${\rm PSL}(3, \H)$. Further classifying elements $g$ in ${\rm SL}(3, \H)$ such that $hgh^{-1}=-g^{-1}$ for some $h \in {\rm SL}(3, \H)$, we obtain a complete classification of reversible elements in ${\rm PSL}(3, \H)$. In particular, we prove the following theorem. \\
  
\begin{theorem}\label{thm-main-equiv-PSL-1}
An element of $ \mathrm{PSL}(3,\mathbb{H})$ is reversible if and only if it is strongly reversible.
\end{theorem}

 \thmref{thm-rev-GL(3,H)}  and 	\lemref{lem-rev-PSL-type-2} will give a complete classification of reversible elements in the group $ \mathrm{PSL}(3,\mathbb{H})$.	 After we have classified reversibility in ${\rm PSL}(3, \H)$,  we can have a broad algebraic classification of the elements of ${\rm PSL}(3, \H)$ as follows.  We refer to \defref{def-embedd-phi}  for the  embedding $\Phi$ from ${\rm GL}(3, \H)$ to ${\rm GL}(6, \C)$.\\

\begin{corollary}\label{cor-rev-selfdual-char}
Suppose $[g]$ is an element of ${\rm PSL}(3,\H)$, where $g$ is a lift of $[g]$ in ${\rm SL}(3,\H)$. Let $\Phi(g)$ be the embedding of $g$ in ${\rm SL}(6,\C)$ and let $\chi_{\Phi(g)}$ be its characteristic polynomial, given by:
$$\chi_{\Phi (g)}=x^6-c_5 x^5 + c_4 x^4 -c_3 x^3+c_2 x^2 -c_1 x +1,$$
	where all the coefficients are real numbers.  Then $g$ is reversible in  ${\rm SL}(3, \H)$ if and only if $c_5=c_1$ and $c_4=c_2$. 
\end{corollary} 

The coefficients $c_i$ of the characteristic polynomial in the above corollary can be expressed in terms of the entries of $g$; see  \cite{KL}.  Thus, elements of ${\rm PSL}(3, \H)$ can be classified into two broad classes: reversible and non-reversible.  For reversible elements in ${\rm PSL}(3, \H)$, one can algebraically classify different subclasses by using the same method as in \cite{CG}. \\

There is another approach to classify the elements of ${\rm SL}(2, \H)$ in terms of some numerical invariants which are not necessarily conjugacy invariants, but up to conjugacy, one can fix those invariants for certain representatives of the conjugacy classes;  see  \cite{Ca}, \cite{PS}, \cite{Wi}. Parker and Short used such invariants to relate an arbitrary element of ${\rm SL}(2, \H)$ to elements of simpler types, mostly real matrices, and then to classify them algebraically using other conjugacy invariants. In the second part of this paper, we follow this approach and decompose any arbitrary element of ${\rm SL}(3, \H)$ into products of {\it simple elements}. \\

\begin{definition}\label{def-simple-element}
	An element $g$ of ${\rm SL}(3, \H)$ is called \emph{simple} if it is conjugate to an element of ${\rm SL}(3, \R)$. 
\end{definition}

Here, ${\rm SL}(3, \R)$ is a component of the subgroup of ${\rm SL}(3, \H)$ consisting of matrices with real entries and determinant  $1$. It is identified with the usual special linear group over the real numbers. We prove the following theorem.\\

\begin{theorem}\label{th:simple4} Every element of ${\rm SL}(3,\H)$ can be written as a product of four simple elements.
\end{theorem}

In Table \ref{table:1}, we have summarized the number of simple elements required to express different Jordan forms in $ \mathrm{SL}(3, \mathbb{H})$ as products of simple elements in $\mathrm{SL}(3,\H)$.
\begin{table}[ht]
	\centering{
		\caption{Simplicity of Jordan forms in $\mathrm{SL}(3,\mathbb{H})$}
		\begin{tabular}{|p{2.7cm}|p{1.6cm}|p{2.7cm}|p{1.6cm}|p{2.7cm}|p{1.6cm}|}
			\hline
			\multicolumn{6}{|c|}{\textbf{Different Jordan forms and their simplicity  in} $\mathrm{SL}(3,\mathbb{H})$} \\
			\hline 
			Jordan form & Number of  simple matrices & Jordan form & Number of  simple matrices & Jordan form & Number of  simple matrices \\
			\hline
			$\begin{pmatrix}
			e^{\ib \theta} & 0 & 0\\0 & e^{\ib\phi} & 0\\0 & 0 & e^{\ib\psi}
			\end{pmatrix}$ & \hspace{.6cm} $3$ & $\begin{pmatrix}
			e^{\ib \theta} &  1 & 0\\0 & e^{\ib \theta} & 0\\0 & 0 & e^{\ib\psi}
			\end{pmatrix}$ &\hspace{.6cm}  $3$ & $\begin{pmatrix}
			e^{\ib \theta} &  1 & 0\\0 & e^{\ib \theta} &  1\\0 & 0 & e^{\ib \theta}
			\end{pmatrix}$  & \hspace{.6cm} $4$\\ 
			\hline
			$\begin{pmatrix}
			\lambda & 0 & 0\\0 & \mu & 0\\0 & 0 & \xi	\end{pmatrix}$	 & \hspace{.6cm} $4$ & $\begin{pmatrix}
			\lambda & 1 & 0\\0 & \lambda & 0\\0 & 0 & \xi	\end{pmatrix}$ & \hspace{.6cm} $4$ &  & \\
			\hline
		\end{tabular}\label{table:1}}
\end{table}

One can algebraically classify simple matrices of ${\rm SL}(3, \H)$. For this, one can use ideas from Goldman's classification \cite[Section 1.7]{Go2}, where hyperbolic elements in ${\rm SL}(3, \R)$ have been classified using the \textit{discriminant} of the characteristic polynomial. We refined the elements of ${\rm SL}(3, \R)$ using their Jordan forms and refined their dynamical classification; see  Definition \ref{def:classification}. Then we extend Goldman's classification to all elements of ${\rm SL}(3, \R)$. We prove the following result.\\

\begin{theorem} \label{th:class1} 
Suppose $A\in  \mathrm{SL}(3,\mathbb{R})$ and  define the function $$f(x,y)=-x^2y^2+4(x^3+y^3)-18xy+27,$$ where $ x=\mathrm{tr}(A)$ and $y=\mathrm{tr}(A^{-1})$. Let $d$ be the maximum degree among the different factors of the minimal polynomial of $A$. Then the following statements are true.
\begin{enumerate}[(i)]
			\item $A$ is regular loxodromic if and only if $f(x,y)>0$.
			\item $A$ is regular elliptic if and only if $f(x,y)<0$ and $x=y$.
			\item $A$ is screw loxodromic if and only if $f(x,y)<0$ and $x\neq y$.
			\item $A$ is homothety if and only if $f(x,y)=0$, $x\neq y$ and $d=1$.
			\item $A$ is loxo-parabolic if and only if $f(x,y)=0$, $x\neq y$ and $d \neq 1$.
			\item $A$ is reflection (or elliptic-reflection) if and only if $f(x,y)=0,\; x=y$ and $d =1$.
			\item $A$ is unipotent if and only if $f(x,y)=0,\;x=y=3$ and $d \neq 1$.
			\item $A$ is ellipto-parabolic if and only if $f(x,y)=0,\;x=y \neq 3$ and $d \neq 1$.
\end{enumerate}
	
\end{theorem}
In Table \ref{table:2}, we have summarized the results obtained in	\thmref{th:class1}.
	\begin{table}[ht]
		\centering{
			
			\caption{Classification of the elements of $\mathrm{SL}(3,\mathbb{R})$}	
			\begin{tabular}{|p{2.1cm}|p{2.2cm}|p{2.3cm}|p{2.1cm}|p{1.9cm}|p{2.2cm}|}
				\hline
				
				\multicolumn{6}{|c|} {\textbf{Classification of} $A\in \mathrm{SL}(3,\mathbb{R})$} \\
				\hline 
				\multicolumn{2}{|c|}{\textbf{Loxodromic}} &  \multicolumn{2}{|c|}{\textbf{Parabolic}} & \multicolumn{2}{|c|}{\textbf{Elliptic}} \\
				\hline
				Regular loxodromic & $f(x,y)>0$ & Vertical translation (unipotent)&  $f(x,y)=0$, $x=y=3$, $d =2$ & Regular elliptic & $f(x,y)<0$, $x=y$ \\ 
				\hline
				Screw loxodromic & $f(x,y)<0$, $x\neq y$ & Non-vertical translation (unipotent) &  $f(x,y)=0$, $x=y=3$, $d =3$ & Reflection & $f(x,y)=0$, $x=y$, $d =1$ \\ 
				\hline
				Homothety & $f(x,y)=0,$ $x\neq y$, $d =1$  &   Ellipto-parabolic &$f(x,y)=0$, $x=y\neq 3$,  $d =2$ &&   \\
				\hline 
				Loxo-parabolic &$f(x,y)=0$, $x\neq y$, $d=2$  &  &  & & \\
				
				\hline
			\end{tabular}\label{table:2}}
	\end{table}
	
\textbf{Structure of the paper.} 
The paper is organized as follows. In Section $\ref{sec-prel}$, we fix some notation and recall necessary background. We discuss some useful preliminary results in Section $\ref{sub-preliminary-results}$. In Section $\ref{sec-rev}$, we investigate the reversibility problem in the group ${\rm SL}(3, \H)$. We prove Theorem $\ref{thm-main-equiv-PSL-1}$, which establishes the equivalence between reversible and strongly reversible elements of the projective linear group ${\rm PSL}(3, \H)$, in Section $\ref{sec-rev-projective-group}$. Section $\ref{sec-decom}$ deals with the decomposition of ${\rm SL}(3, \H)$ into simple elements, and we prove Theorem $\ref{th:simple4}$. We give the classification of the elements of ${\rm SL}(3, \R)$ in Section $\ref{sec:sl3r}$, and finally, we summarize the classifications in Section $\ref{sec-final-class}$. 

\section{Preliminaries}\label{sec-prel}
In this section, we fix some notation and recall some necessary background that will be used throughout this paper. Let $\H:= \R + \R \ib + \R \jb + \R \kb$ denote the division algebra of Hamilton’s quaternions, where $\ib^2=\jb^2=\kb^2=\ib  \jb   \kb = -1$.   For an elaborate discussion on the linear algebra over the quaternions;  see  \cite{rodman},  \cite{FZ}.

\begin{definition}\label{def-eigen-M(n,H)}
Let $A \in  \mathrm{M}(3,\H)$, the algebra of $3 \times 3$ matrices over $\H$.  A non-zero vector $v \in \H^3 $ is said to be a (right) eigenvector of $A$ corresponding to a  (right) eigenvalue  $\lambda \in \H $ if the equality $ A v = v\lambda $ holds.
\end{definition}

Eigenvalues of $A$ occur in similarity classes, and each similarity class of eigenvalues contains a unique complex number with non-negative imaginary part. Here, instead of similarity classes of eigenvalues, we will consider the \textit{unique complex representatives} with non-negative imaginary parts as eigenvalues unless specified otherwise. 

\begin{definition}\label{def-embedd-phi}
Let $A \in  \mathrm{M}(3,\H)$ and write $ A = A_1 + A_2 \jb $,  where $ A_1,  A_2 \in  \mathrm{M}(3,\C)$.  We denote the \textit{complex adjoint} of $A$ by $\Phi(A)$, which is defined as follows:
	$$ \Phi(A)=  \begin{pmatrix} A_1   &  A_2 \\
		- \overline{A_2} & \overline{A_1}  \\ 
	\end{pmatrix}.$$
Note that  the map $A \longmapsto \Phi(A) $ is an embedding from  $\mathrm{M}(3,\H)$ to $\mathrm{M}(6,\C)$;  see \cite{rodman}, \cite{FZ}.  
\end{definition}
\begin{definition}\label{def:sdet}
	For $A \in  \mathrm{M}(3,\H)$, the determinant (resp.  characteristic polynomial) of $A$ is denoted by $ {\rm det_{\H}(A)}$ (resp.  $\chi_{\H}(A)$) and given  by the  determinant (resp.  characteristic polynomial) of the corresponding complex adjoint matrix $ \Phi(A)$, i.e., $ {\rm det_{\H}(A)}:=  {\rm det(\Phi(A))}$ and $\chi_{\H}(A) := \chi_{\Phi(A)}$; see   \cite{FZ}.  
\end{definition} 

Consider the Lie groups 
$\mathrm{GL}(3,\H) := \{ g \in  \mathrm{M}(3,\H) \mid {\rm det_{\H}}(g) \neq 0 \}$ and 
$  \mathrm{SL}(3,\H) := \{ g \in \mathrm{GL}(3,\H) \mid {\rm det_{\H}}(g) = 1 \}$. Note that $ {\rm det_{\H}}(A)  \in \R \ \hbox{and} \  {\rm det_{\H}}(A) \geq 0 $ for all $A  \in  \mathrm{M}(3,\H)$; see  {\cite[Proposition 4.2]{FZ}} and  {\cite [Theorem 5.9.2.(6)] {rodman}}. 
Recall  the Jordan decomposition for matrices over $\H$,  which gives conjugacy classification of elements of $\mathrm{GL}(3,\mathbb{H})$; see {\cite[Theorem 5.5.3]{rodman}}.   More precisely, up to conjugacy,  we can assume that every element of $\mathrm{GL}(3,\mathbb{H})$  is one of the following matrices:
\begin{enumerate}[(i)]
	\item $ \begin{pmatrix}
		r e^{\ib \theta} &0 &0 \\
		0 & s e^{\ib \phi}& 0\\
		0 & 0 & te^{\ib \psi}
	\end{pmatrix} $, where $r,s,t \in \mathbb{R}^{+}:= \{x\in \R \mid x >0\}$ and $\theta, \phi, \psi \in [0,\pi]$,

\vspace{3mm}
	\item $\begin{pmatrix}
		r e^{\ib \theta} &1 &0 \\
		0 & r e^{\ib \theta}& 0\\
		0 & 0 & t e^{\ib \psi}
	\end{pmatrix}$, where $r, t \in \mathbb{R}^{+}$ and $\theta,\psi \in [0,\pi]$,

\vspace{3mm}
	
	\item $\begin{pmatrix}
		r e^{\ib \theta} &1 &0 \\
		0 & r e^{\ib \theta}& 1\\
		0 & 0 & r e^{\ib \theta}
	\end{pmatrix}$,  where $r \in \mathbb{R}^{+}$ and $\theta \in [0,\pi]$.
\end{enumerate}

\begin{remark}\label{rem-conjugacy-SL(n,H)} Let $g \in \mathrm{GL}(3,\mathbb{H})$ such that $gAg^{-1} =B$,  where $A,B \in \mathrm{M}(3,\mathbb{H})$.
	In view of   {\cite[Proposition 4.2]{FZ}}, we have ${\rm det_{\H}}(g)  \in \R$ and $ {\rm det_{\H}}(g)  >0 $.  Let $ a:= \frac{1}{({\rm det_{\H}}(g))^{1/3}}$. Then $ a{\rm I}_3$ lies in the centre of ${\GL}(3,\H)$. Consider $h:= g(a{\rm I}_3)$.  Then $h \in \mathrm{SL}(3,\mathbb{H})$ such that $hAh^{-1}= B$.  Thus, if two elements of $\mathrm{M}(3,\mathbb{H})$ are conjugate in $\mathrm{GL}(3,\mathbb{H})$, then without loss of generality, we can assume that they are conjugate in $\mathrm{SL}(3,\mathbb{H})$.
\end{remark}

Consider the  projective linear group $\mathrm{PSL}(3,\H) := \mathrm{SL}(3,\H)/ \{\pm \mathrm{I}_3\}$, which acts on the quaternionic projective space  $ \P_{\H}^2$.  In view of the Jordan decomposition of matrices over $\H$,  we have the following well-defined classification for $\mathrm{PSL}(3,\H)$ into three mutually exclusive classes.

\begin{definition}\label{def:classification}
	Let $ [g] \in \mathrm{PSL}(3,\H)$ be a projective transformation, where $g$ is a lift of $[g]$ in $\mathrm{SL}(3,\H)$. Then
	\begin{enumerate}[(i)]
		\item  $[g]$ is called elliptic if  $g$ is (semisimple) diagonalizable and each eigenvalue of $g$ has a unit modulus.
		
		\item  $[g]$ is called loxodromic if $g$  has at least one eigenvalue with a non-unit modulus.
		
		\item  $[g]$ is parabolic if $g$ is non-diagonalizable and each eigenvalue of $g$ has a unit modulus.
	\end{enumerate}
\end{definition}

We can further divide the above classes into more specific subclasses. An elliptic element is \textit{regular elliptic} if all its eigenvalues are distinct. Otherwise, we call it \textit{elliptic-reflection}.
A parabolic element is\textit{ unipotent} if all its eigenvalues are equal to $1$. Otherwise, we call it \textit{ellipto-parabolic} or \textit{ellipto-translation}, depending on whether the corresponding minimal polynomial has a factor of degree $2$ or $3$. Unipotent elements of $\mathrm{SL}(3,\H)$ can be divided into two cases depending on their minimal polynomials. A unipotent element is called \textit{vertical translation} (resp. \textit{non-vertical translation}) if the corresponding minimal polynomial is $(x - 1)^2$ (resp. $(x - 1)^3$). A loxodromic element is \textit{regular loxodromic} if all its eigenvalues have distinct moduli. A loxodromic element is called a \textit{screw loxodromic} if two of its eigenvalues are not equal but have the same modulus. A loxodromic element with two equal eigenvalues is either \textit{homothety} or \textit{loxo-parabolic}, depending on whether it can be diagonalized or not. It is worth noting that the naming conventions for these subclasses have been influenced by similar names found in \cite{dgt}.

\subsection{Preliminary results}\label{sub-preliminary-results}
In this subsection, first, we recall some well-known primary results which follow from the fundamental properties of quaternions. We include their proofs for the sake of completeness. These results are helpful in studying reversibility in $\mathrm{SL}(3,\H)$ and assist in identifying elements of $\mathrm{SL}(3,\H)$ which are reversible but not strongly reversible; see \lemref{lem-non-strong-rev-GL(3,H)}. 

\begin{lemma}  [cf.~{\cite [Lemma 2.3.]{GL}}]\label{lem-ref- lAA- paper}
	Let $\theta \in [0,\pi ]$ and $ \alpha, \beta  \in (-\pi,\pi )$. Then the following statements are true.  
	\begin{enumerate}
		
		\item  \label{lem-basic-reverser-calculation-part-1} 	Let  $a \in \H$ be such that $ a e^{\ib \theta} = e^{-\ib \theta} a $. Then $a \in \begin{cases}
			\C \,  \jb & \text{if  $\theta \neq 0,  \pi$};\\
			\H  & \text{if $ \theta = 0,  \pi $}.
		\end{cases}$
		
		\item Let  $a \in \H$ be such that $ a e^{\ib \theta} = e^{\ib \theta} a $.  Then $a \in   \begin{cases}
			\C \,  & \text{if  $\theta \neq 0,  \pi$};\\
			\H & \text{if $ \theta = 0,  \pi $}.
		\end{cases}$
		
		\item 	Let  $b \in \H$ be such that $b e^{\ib \alpha}= e^{\ib \beta } b$.  If $ \alpha  \neq  \pm \beta $ then $b=0$.
	\end{enumerate}
\end{lemma}
\begin{proof}
	Here, we will only prove the first part of the lemma. Using similar arguments, one can prove other parts of the lemma.
	Let  $a = z + w \jb$ be such that $a \,  e^{\ib \theta} =  e^{-\ib \theta} \, a$, where $z,w \in \C$.  This implies
	\begin{equation}\label{eq-proof-1-reverser}
		(z + w \jb) \,  e^{\ib \theta} =  e^{-\ib \theta} \,  (z + w \jb).
	\end{equation}
	On comparing both sides of  Equation \eqref{eq-proof-1-reverser} and using $\jb \bar{x} = x \jb$ for all $x \in \C$,  we have
	\begin{equation}
		z e^{\ib \theta} =  e^{-\ib \theta} z  \hbox{ and } w e^{-\ib \theta} =  e^{-\ib \theta} w.
	\end{equation}
	Due to  commutativity of complex numbers, $w e^{-\ib \theta} =  e^{-\ib \theta} w$ holds for all $ w \in \C$.  The proof of \eqref{lem-basic-reverser-calculation-part-1}  now follows from the equation $z e^{\ib \theta} =  e^{-\ib \theta} z $.
\end{proof}

\begin{lemma}\label{lem-basic-reverser-calculation}
	The following statements are true.
	\begin{enumerate}[(i)]
		\item \label{lem-basic-reverser-calculation-part-2} 	Let $\mathcal{R}(e^{\ib \theta}) := \{ a \in  \H \setminus \{0\} \mid a e^{\ib \theta} = e^{-\ib \theta} a \}$, where $\theta \in [0,\pi ]$. Let $ a \in \mathcal{R}(e^{\ib \theta}) \, \cup \, \{0\}$  and $b \in  \H $ such that 
		$a + b \, (e^{\ib \theta}) = (e^{-\ib \theta}) \,  b $.  Then $a =0$ and $ b \in \mathcal{R}(e^{\ib \theta}) \, \cup \, \{0\}$.
		
		\item \label{lem-basic-reverser-calculation-part-3} Let $ a \, (r e^{\ib \theta}) = (s e^{\ib \phi}) \,a$, where $r,s \in \R^{+}:=\{ x \in  \R \mid x >0 \} $, $a \in \H$ and $ \theta, \phi \in [0,\pi ]$.  If  $r \neq s$, then $a =0$.  
	\end{enumerate}
\end{lemma}
\begin{proof}
	\textit{Proof of (\ref{lem-basic-reverser-calculation-part-2}).}  If $ \theta \in \{0, \pi\}$ then the result holds trivially.  So we assume $  \theta \in (0, \pi)$.
	As $ a \in \mathcal{R}(e^{\ib \theta}) \, \cup \, \{0\}$ and $  \theta \in (0, \pi)$,  so from  \lemref{lem-ref- lAA- paper},  we have $a = z \jb$ for some $ z \in \C$.  Let  $b = u + v \jb$ be such that $a + b \, (e^{\ib \theta}) = (e^{-\ib \theta}) \,  b $, where $u,v \in \C$.   This implies 
	\begin{equation}\label{eq-proof-2-reverser}
		z\jb + (u + v \jb) \,  e^{\ib \theta} =  e^{-\ib \theta} \,  (u + v \jb).
	\end{equation}
	Using $\jb \bar{x} = x \jb$ for all $x \in \C$ and on comparing both sides of  Equation \eqref{eq-proof-2-reverser}, we  have:
	\begin{equation}\label{eq-part2-1-reverser}
		u  e^{\ib \theta}  =  e^{-\ib \theta} u, \hbox { i.e., } u  e^{\ib \theta}  =  u e^{-\ib \theta}, \hbox { and} 
	\end{equation}
	\begin{equation} \label{eq-part2-2-reverser}
		z +	v  e^{-\ib \theta}  =  e^{-\ib \theta} v,  \hbox { i.e., }  z + v  e^{-\ib \theta}  =  ve^{-\ib \theta}.
	\end{equation}
	The proof now follows from Equation \eqref{eq-part2-1-reverser} and Equation \eqref{eq-part2-2-reverser}.
	
	\textit{Proof of  (\ref{lem-basic-reverser-calculation-part-3}).} We have $ a (r e^{\ib \theta}) = (s e^{\ib \phi}) a$, where $a \in \H$.
	As 	$r,s >0$ and  $r \neq s$, so by taking the modulus on both sides of the above equation,  we get  $a =0$.   
\end{proof}

Now, we will investigate the simple elements of $\mathrm{SL}(3,\mathbb{H})$. Recall that simple elements of $\mathrm{SL}(3,\mathbb{H})$ are those elements which are conjugate to a real matrix. Note the following result from \cite{rodman}. 

\begin{lemma}  [cf.~{\cite  [Theorem 5.8.1 (b)]{rodman}}]\label{lemma-rod-simple}
	A square-size quaternion matrix $A$ is similar to a real matrix if and only if for every $m \in \N$ and every non-real eigenvalue $
	\lambda \in \H$, the number of Jordan blocks $\mathrm{J}(\alpha, m)$ with $
	\alpha \in \H$ similar to $\lambda$ in the Jordan form of $A$ is even.
\end{lemma}

The following result follows from \lemref{lemma-rod-simple}, which classify the simple elements of $\mathrm{SL}(3,\mathbb{H})$.

\begin{proposition} \label{prop-simple-classification}
	An element $A \in \mathrm{SL}(3,\mathbb{H})$ is conjugate to an element $B \in  \mathrm{SL}(3,\mathbb{R})$ if and only if its eigenvalues are either all real or it is diagonalizable with only one non-real eigenvalue of multiplicity two.
\end{proposition}

Note the following immediate Corollary of  \propref{prop-simple-classification}.

\begin{corollary} \label{cor-simple-uppertriangular-real-eigenvalue}
	Let $A \in \mathrm{SL}(3,\mathbb{H})$ be an upper triangular matrix with all the diagonal entries being real numbers. Then $A$ is a simple matrix.
\end{corollary}

\begin{remark}	\label{rem-simple-complex-eigenvalues}
	In view of \propref{prop-simple-classification},  if $A \in \mathrm{SL}(3,\mathbb{H})$ is simple and has a non-real eigenvalue, then up to conjugacy,  $A$ is conjugate to  $B=\mathrm{diag}(re^{\ib \theta}, re^{\ib \theta}, r^{-2}e^{\ib \psi})$, where $r \in \mathbb{R}^{+}$, $\theta \in (0,\pi)$ and $\psi \in \{0,\pi\}$. To see that matrix $B$ is conjugate to a real matrix, note the following equations:
	\begin{align*}
		&	\begin{pmatrix}
			1 & 0 \\ 0 & \jb 
		\end{pmatrix} \begin{pmatrix}
			r e^{\ib \theta} & 0 \\0& r e^{\ib \theta}
		\end{pmatrix}  =  \begin{pmatrix}
			r e^{\ib \theta} & 0 \\0& r e^{-\ib \theta}
		\end{pmatrix}	\begin{pmatrix}
		1 & 0 \\ 0 & \jb 
	\end{pmatrix}, \hbox{ and } \\
		&	\begin{pmatrix}
			1 & 1 \\ \ib & -\ib 
		\end{pmatrix} \begin{pmatrix}
			r e^{\ib \theta} & 0 \\0& r e^{-\ib \theta}
		\end{pmatrix}=  \begin{pmatrix}
			r \cos \theta & r \sin \theta \\- r \sin \theta & r \cos \theta
		\end{pmatrix} \begin{pmatrix}
			1 & 1 \\ \ib & -\ib 
		\end{pmatrix}.
	\end{align*}
\end{remark}
	
\section{Reversibility in $\mathrm{SL}(3,\mathbb{H})$}\label{sec-rev}
In this section, we will investigate the reversibility problem in the group $\mathrm{SL}(3,\mathbb{H})$. Before that,  it is worth noting that if $A \in \mathrm{GL}(3,\H)$ is reversible, then  $A$ is an element of $ \mathrm{SL}(3,\H)$.  Moreover, the next two results establish that it is enough to restrict our attention to 
$\mathrm{SL}(3,\mathbb{H})$ to investigate reversibility.

\begin{lemma}\label{lem-rev-SL(3,H)}
		An element $A$ of $\mathrm{SL}(3,\H)$  is reversible in $\mathrm{SL}(3,\H)$  if and only if it is reversible in $\mathrm{GL}(3,\H)$.
\end{lemma} 
	
\begin{proof}
To see this, suppose that $gAg^{-1}=A^{-1}$, where $g \in \mathrm{GL}(3,\H)$. Either $g \in \mathrm{SL}(3,\H)$ or $g \not\in \mathrm{SL}(3,\H)$. In the latter case, using  \remref{rem-conjugacy-SL(n,H)}, we can construct $h \in \mathrm{SL}(3,\H)$ such that $h  Ah^{-1}=A^{-1}$.  Hence, the proof follows.
\end{proof}
	
The classification of strongly reversible elements in $ \mathrm{SL}(3,\H)$   differs from that in $ \mathrm{SL}(3,\C)$.   There are reversible elements in $\mathrm{SL}(3,\mathbb{C})$ which are strongly  reversible in  $\mathrm{GL}(3,\mathbb{C})$ but not in $\mathrm{SL}(3,\mathbb{C})$.  However, for $A \in \mathrm{SL}(3,\mathbb{H})$, strong reversibility in $\mathrm{SL}(3,\mathbb{H})$ and $\mathrm{GL}(3,\mathbb{H})$ is equivalent.

\begin{lemma}
Let $A \in \mathrm{SL}(3,\H)$.  Then $A$ is strongly reversible in $ \mathrm{SL}(3,\H)$  if and only if it is strongly reversible in $ \mathrm{GL}(3,\H)$.
\end{lemma}

\begin{proof}
To see this,  suppose that $gAg^{-1}=A^{-1}$,  where $g \in \mathrm{GL}(3,\H)$ is an involution.  Note that $\Phi(g) \in \mathrm{GL}(6,\C)$ is an involution,  so ${\rm det(\Phi(g))} = \pm 1$.  Since $\mathrm{det}_{\H}(g)$ is always non-negative, we have $ {\rm det_{\H}(g)}:=  {\rm det(\Phi(g))} = 1$, i.e.,  $g \in \mathrm{SL}(3,\H)$.  
\end{proof}
	
The following theorem classifies reversible elements in $\mathrm{SL}(3,\mathbb{H})$.
\begin{theorem}\label{thm-rev-GL(3,H)}
An element $A \in \mathrm{SL}(3,\mathbb{H})$ is reversible in $ \mathrm{SL}(3,\H)$  if and only if it is conjugate in $\mathrm{SL}(3,\mathbb{H})$ to one of the following matrices:
\begin{enumerate}[(i)]
			\item \label{rev-type-1} 
			$\begin{pmatrix}
			e^{\ib \theta} &0 &0 \\
			0 &  e^{\ib \phi}& 0\\
			0 & 0 & e^{\ib \psi}
			\end{pmatrix}$,  where $\theta, \phi, \psi \in [0,\pi]$,
			
			\vspace{3mm}
			\item \label{rev-type-2}
			$\begin{pmatrix}
			re^{\ib \theta} &0 &0 \\
			0 &  r^{-1}e^{\ib \theta}& 0\\
			0 & 0 & e^{\ib \psi}
			\end{pmatrix}$,  where $r \in \R^{+}, r \neq 1$,  and $\theta, \psi \in [0,\pi]$,
			\vspace{3mm}
			\item \label{rev-type-3}
			$\begin{pmatrix}
			e^{\ib \theta} &1 &0 \\
			0 &  e^{\ib \theta} & 0\\
			0 & 0 & e^{\ib \psi}
			\end{pmatrix}$,  where $\theta, \psi \in [0,\pi]$,
			
			\vspace{3mm}
			
			\item  \label{rev-type-4}
			$\begin{pmatrix}
			e^{\ib \theta} &1 &0 \\
			0 &  e^{\ib \theta} & 1\\
			0 & 0 & e^{\ib \theta}
			\end{pmatrix}$,  where $\theta \in [0,\pi]$.
\end{enumerate}
\end{theorem}
\begin{proof} We can assume, without loss of generality, that $A$ is equal to one of the matrices given in $(\ref{rev-type-1})-(\ref{rev-type-4})$. To show that $A$ is reversible in $\mathrm{SL}(3,\mathbb{H})$, it is sufficient to find an element $g$ in $\mathrm{SL}(3,\mathbb{H})$ such that $gAg^{-1}=A^{-1}$. Note that for the matrices given in $(\ref{rev-type-1})-(\ref{rev-type-4})$, we can choose a desired $g\in\mathrm{SL}(3,\mathbb{H})$ in the following way, respectively:
\begin{enumerate}[(a)]
			\item\label{rev-type-1-skew}  $ g:= \begin{pmatrix}
			\jb & 0 &0 \\
			0 &  \jb & 0\\
			0 & 0 & \jb 
			\end{pmatrix}, $
		
			\vspace{3mm}
			
			\item \label{rev-type-2-skew}  $g:= \begin{pmatrix}
			0 & \jb &0 \\
			\jb &  0 & 0\\
			0 & 0 & \jb 
			\end{pmatrix},$
		
			\vspace{3mm}
			
			\item \label{rev-type-3-skew} $g:=    \begin{pmatrix}
			- (e^{-2\ib \theta})\,  \jb&0 &0 \\
			0 &  \jb & 0\\
			0 & 0 & \jb \\
			\end{pmatrix},$
		
			\vspace{3mm}
			
			\item \label{rev-type-4-skew} $g:=  \begin{pmatrix}
			(e^{-4\ib \theta})\,  \jb & (e^{-3\ib \theta})\,  \jb &0 \\
			0 &  - (e^{-2\ib \theta})\,  \jb  & 0\\
			0 & 0 & \jb \\
			\end{pmatrix}.$ 
\end{enumerate}
		
Conversely,  let $g A g^{-1} = A ^{-1}$  for some $g \in \mathrm{SL}(3,\mathbb{H})$. In view of the Jordan decomposition over $\H$, $A$ is conjugate to $A^{-1}$ if and only if $A$ and $A^{-1} $ have same Jordan form.  Now, recall the fact that for a unique complex representative $\lambda$ of an eigenvalue class of  $A$,  $[\lambda]= [\lambda^{-1}]$ if and only if $|\lambda|= 1$, i.e.,  $\lambda^{-1} = \overline{\lambda}$.  This implies  that the  blocks  in the Jordan form of $A$ can be partitioned into pairs $ \{ \mathrm{J}(\alpha, s),\mathrm{J}(\alpha^{-1}, s)\} $, or singletons $\{\mathrm{J}(\beta, t  )\}$,  where $\alpha, \beta \in \C \setminus \{0\}$  with non-negative imaginary part such that  $|\alpha| \neq 1,  |\beta| = 1$.  Hence, $A$ will be conjugate to one of the matrices in $(\ref{rev-type-1})-(\ref{rev-type-4})$. This completes the proof.
\end{proof}

Observe that the characteristic polynomial $\chi_{\H}(A)$ is \textit{self-dual} for each reversible element $A \in \mathrm{SL(3,\H)}$;  see \defref{def:sdet}.   
From this, the proof of  \corref{cor-rev-selfdual-char} follows.

\textbf{Proof of  \corref{cor-rev-selfdual-char}.} Consider the embedding $\Phi$ from $\mathrm{SL(3,\H)}$  to $\mathrm{SL(6,\C)}$ which is defined in \defref{def-embedd-phi}. The proof now follows from \thmref{thm-rev-GL(3,H)}.
\qed

\begin{remark} \label{remark_non-rev-SL(3,H)}
	Let $g$ be a non-reversible element of $\mathrm{SL(3,\H)}$. Then  up to conjugacy, $g$  has one of the following forms:
	\begin{enumerate}
		\item $ \begin{pmatrix}
			r e^{\ib \theta} &0 &0 \\
			0 & s e^{\ib \phi}& 0\\
			0 & 0 & \frac{1}{rs}e^{\ib \psi}
		\end{pmatrix} $, where $r,s \in \mathbb{R}^{+} \setminus \{1\}$  and $\theta, \phi, \psi \in [0,\pi]$ such that either $rs \neq 1$ or $\theta \neq \phi$,
	
		\vspace{3mm}
		
		\item $\begin{pmatrix}
			r e^{\ib \theta} &1 &0 \\
			0 & r e^{\ib \theta}& 0\\
			0 & 0 &  \frac{1}{r^2} e^{\ib \psi}
		\end{pmatrix}$, where $r \in \mathbb{R}^{+}$, $r \neq 1$ and $\theta,\psi \in [0,\pi]$.
	\end{enumerate}
\end{remark}

An element $g \in \mathrm{SL(3,\H)}$ is called a \textit{skew-involution} if $g^2 = - \mathrm{I}_3$.   Note that the conjugating elements given by $\eqref{rev-type-1-skew} -\eqref{rev-type-4-skew}$ in the proof of \thmref{thm-rev-GL(3,H)},  are skew-involutions. Therefore, we have the following result.

\begin{proposition}\label{prop-rev-prod-skew-inv}
	Let $A \in  \mathrm{SL}(3,\H)$ be a reversible element. Then $A$ can be written as a product of two skew-involutions in $\mathrm{SL}(3,\H)$.
\end{proposition}

\begin{proof}
	Recall that an element $A$ in $\mathrm{SL}(3,\H)$ can be written as a product of two \textit{skew-involutions} if and only if there exists $g \in \mathrm{SL}(3,\H)$ such that $gAg^{-1} = A^{-1}$ and $g^2= - \mathrm{I}_3$. Without loss of generality, we can assume that $A$ is equal to one of the matrices listed in $(\ref{rev-type-1})- (\ref{rev-type-4})$ of \thmref{thm-rev-GL(3,H)}. We choose $g$ as defined in the proof of \thmref{thm-rev-GL(3,H)}, which is given by $\eqref{rev-type-1-skew} - \eqref{rev-type-4-skew}$, respectively. Then, $g$ is a skew-involution such that $gAg^{-1} = A^{-1}$. The proof now follows from \thmref{thm-rev-GL(3,H)}.
\end{proof}

The following lemma gives us strongly reversible elements in $ \mathrm{SL}(3,\mathbb{H})$.

\begin{lemma} \label{lem-strong-rev-GL(3,H)}
	Suppose $A \in \mathrm{SL}(3,\mathbb{H})$ is  one of the following matrices:
	\begin{enumerate}[(i)]
		\item\label{strong-rev-type-1} $\begin{pmatrix}
			e^{\ib \theta} &0 &0 \\
			0 &  e^{\ib \theta}& 0\\
			0 & 0 & e^{\ib \psi}
		\end{pmatrix}$,  where $\theta \in [0,\pi]$ and $\psi \in \{0,\pi\}$,
		
		\vspace{3mm} 
		
		\item \label{strong-rev-type-2}$\begin{pmatrix}
			r 	e^{\ib \theta} &0 &0 \\
			0 &  r^{-1} 	e^{\ib \theta} & 0\\
			0 & 0 & e^{\ib \psi}
		\end{pmatrix}$,  where $r \in \R^{+}, r \neq 1$,  and $\psi \in \{0,\pi\}$,
		\vspace{3mm}
		\item \label{strong-rev-type-3}$\begin{pmatrix}
			e^{\ib \theta} &1 &0 \\
			0 &  e^{\ib \theta} & 0\\
			0 & 0 & e^{\ib \psi}
		\end{pmatrix}$,  where $\theta, \psi \in \{0,\pi\}$,
		\vspace{3mm}
		\item  \label{strong-rev-type-4}$\begin{pmatrix}
			e^{\ib \theta} &1 &0 \\
			0 &  e^{\ib \theta} & 1\\
			0 & 0 & e^{\ib \theta}
		\end{pmatrix}$,  where $\theta \in \{0,\pi\}$.
	\end{enumerate}
	Then $A$ is strongly reversible in $\mathrm{SL}(3,\mathbb{H})$.
\end{lemma}

\begin{proof}
	First note that  $A$ is reversible in $\mathrm{SL}(3,\mathbb{H})$; see \thmref{thm-rev-GL(3,H)}.
	To prove that $A$ is  strongly reversible in  $\mathrm{SL}(3,\mathbb{H})$,  it is sufficient to find an element $g$ in $\mathrm{SL}(3,\mathbb{H})$ such that $gAg^{-1} = A^{-1}$ and $g^2=  \mathrm{I}_3$.   Note that  for  matrices given in $(\ref{rev-type-1})-(\ref{rev-type-4})$,
	we can choose a desired  \textit{involution} $g\in\mathrm{SL}(3,\mathbb{H})$ in the following way, respectively: 
	\begin{enumerate}[(a)]
		\item  $g := \begin{pmatrix}
			0 & \jb &0 \\
			-\jb & 0 & 0\\
			0 & 0 & 1 \\ 
		\end{pmatrix}$,
		\vspace{3mm}
		\item  $g :=  \begin{pmatrix}
			0 & \jb &0 \\
			-	\jb &  0 & 0\\
			0 & 0 & 1 \\ 
		\end{pmatrix}$,
		\vspace{3mm}
		\item   $ g :=  \begin{pmatrix}
			1 &0 &0 \\
			0 &  -1 & 0\\
			0 & 0 & 1 \\
		\end{pmatrix}$,
		\vspace{3mm}
		\item   $g :=   \begin{cases}
			\begin{pmatrix}
				1 & 1 &0 \\
				0 & -1 & 0\\
				0 & 0 & 1 \\
			\end{pmatrix} 	\vspace{3mm} & \text{if  $\theta = 0 $};\\
			
			\begin{pmatrix}
				1 &-1 &0 \\
				0 & -1 & 0\\
				0 & 0 & 1 \\
			\end{pmatrix}  & \text{if  $\theta =   \pi$}.
		\end{cases}$
	\end{enumerate}
Thus, we have an involution $g$ in $\mathrm{SL}(3,\mathbb{H})$ such that $gAg^{-1} = A^{-1}$.  This proves the lemma.
\end{proof}

The next result gives us elements in $\mathrm{SL}(3,\mathbb{H})$ which are reversible but not strongly reversible.
\begin{lemma}\label{lem-non-strong-rev-GL(3,H)}
	Suppose $A \in \mathrm{SL}(3,\mathbb{H})$ is  one of the following matrices:
	\begin{enumerate}
		\item\label{non-strong-rev-type-1} $\begin{pmatrix}
			e^{\ib \theta} &0 &0 \\
			0 &  e^{\ib \theta}& 0\\
			0 & 0 & e^{\ib \theta}
		\end{pmatrix}$,  where $\theta \in (0,\pi)$,
		
		\vspace{3mm} 
		
		\item\label{non-strong-rev-type-2}$\begin{pmatrix}
			e^{\ib \theta} &0 &0 \\
			0 &  e^{\ib \phi}& 0\\
			0 & 0 & e^{\ib \psi}
		\end{pmatrix}$,  where $\theta \neq \phi, \theta \neq \psi$,  $\theta\in (0,\pi)$, and $\phi, \psi \in [0,\pi]$, 
		
		\vspace{3mm} 
		
		\item \label{non-strong-rev-type-3}$\begin{pmatrix}
			re^{\ib \theta} &0 &0 \\
			0 &  r^{-1} e^{\ib \theta}& 0\\
			0 & 0 & e^{\ib \psi}
		\end{pmatrix}$,  where $r \in \R^{+}, r \neq 1$,  $\theta  \in [0, \pi]$ and $ \psi \in (0,\pi)$,
		\vspace{3mm}
		
		\item\label{non-strong-rev-type-5}$\begin{pmatrix}
			e^{\ib \theta} &1 &0 \\
			0 &  e^{\ib \theta} & 0\\
			0 & 0 & e^{\ib \psi}
		\end{pmatrix}$,  where $\theta \in \{0,\pi\}$,  $\psi \in (0,\pi)$,
		
		\vspace{3mm}
		
		\item \label{non-strong-rev-type-6}$\begin{pmatrix}
			e^{\ib \theta} &1 &0 \\
			0 &  e^{\ib \theta} & 0\\
			0 & 0 & e^{\ib \psi}
		\end{pmatrix}$,  where $\theta \neq \psi$ and  $\theta \in (0,\pi)$, $ \psi \in [0,\pi]$,
		
		\vspace{3mm}

		\item \label{non-strong-rev-type-7}$\begin{pmatrix}
			e^{\ib \theta} &1 &0 \\
			0 &  e^{\ib \theta} & 0\\
			0 & 0 & e^{\ib \theta}
		\end{pmatrix}$,  where $\theta \in (0,\pi)$,
		
		\vspace{3mm}
		
		\item  \label{non-strong-rev-type-8}$\begin{pmatrix}
			e^{\ib \theta} &1 &0 \\
			0 &  e^{\ib \theta} & 1\\
			0 & 0 & e^{\ib \theta}
		\end{pmatrix}$,  where $\theta \in (0,\pi)$.

	\end{enumerate}
	Then $A$ is not strongly reversible in $\mathrm{SL}(3,\mathbb{H})$.
\end{lemma}

\begin{proof}
	Let $g =  \begin{pmatrix}
		a_{11} & a_{12}& a_{13} \\
		a_{21} & a_{22} & a_{23}\\
		a_{31} & a_{32} & a_{33} 
	\end{pmatrix}$ in $\mathrm{SL}(3,\mathbb{H})$ such that $gA g^{-1}= A^{-1}$.  Then on comparing both sides of the equation $gA= A^{-1}g$ and  using basic properties of the quaternions given in  \lemref{lem-ref- lAA- paper} and \lemref{lem-basic-reverser-calculation},  for the matrices given in $\eqref{non-strong-rev-type-1} - \eqref{non-strong-rev-type-8}$,  we get that matrix $g$ has the following form, respectively:
	\begin{enumerate}
		\item 
		$g= P \,  \jb $,  where $P \in  \mathrm{GL}(3,\mathbb{C})$,
			\vspace{3mm}
		\item $g= z \,  \jb  \oplus  P $,  where  $z \in \C \setminus \{0\} $ and $P\in  \mathrm{GL}(2,\mathbb{H})$,
			\vspace{3mm}
		\item  $g =  P  \oplus z \, \jb $,  where  $P\in  \mathrm{GL}(2,\mathbb{H})$ and $z \in \C \setminus \{0\}$,
			\vspace{3mm}
		\item  $g =  P  \oplus z\jb $,  where  $P\in  \mathrm{GL}(2,\mathbb{H})$ and $z \in \C \setminus \{0\}$,
		
		\vspace{3mm}
		\item $g =  \begin{pmatrix}
			- (e^{-2\ib \theta} u) \,  \jb & v \, \jb   \\
			0 &  u \, \jb  \\
		\end{pmatrix}\,  \oplus \,  (z) $,  where $u \in \C \setminus \{0\} $, $ v  \in \C$ and $z \in \H \setminus \{0\} $,
		
		\vspace{3mm}
		
		\item $g =  \begin{pmatrix}
			- (e^{-2\ib \theta} u) \,  \jb & v \, \jb & y \, \jb \\
			0 &  u \, \jb & 0\\
			0 & x \,  \jb & z \, \jb  \\
		\end{pmatrix}, $  where  $u,z \in \C \setminus \{0\}$ and $ x,v, y  \in \C$,

		\vspace{3mm}
		
		\item $g =  \begin{pmatrix}
			(e^{-4\ib \theta} z) \,  \jb &  ( - e^{-2\ib \theta} y + e^{-3\ib \theta} z) \, \jb & x \, \jb \\
			0 &  - (e^{-2\ib \theta} z) \,  \jb & y \, \jb \\
			0 & 0 & z \, \jb  \\
		\end{pmatrix}, $  where  $z \in \C \setminus \{0\}$ and $ x,y  \in \C$.
	\end{enumerate}
Now observe that for $ z \in \C \setminus \{0\} $ and  $P \in  \mathrm{GL}(3,\mathbb{C})$, the  following equations  hold:
	\begin{equation} \label{eq-1-non-strong-rev-GL(3,H)}
		(z \jb )^2 \neq 1,  \hbox{ i.e., }  |z|^2 \neq -1, \hbox{ and}
	\end{equation}
		\begin{equation}\label{eq-2-non-strong-rev-GL(3,H)}
		(P \jb)^2 \neq \mathrm{I}_3,  \hbox{ i.e.,  }  \, \mathrm {det}(P^{-1}) \neq \mathrm {det}(- \overline{P}) = -  \overline{\mathrm {det}(P)}.
	\end{equation}
Note that Equation \eqref{eq-1-non-strong-rev-GL(3,H)} and Equation \eqref{eq-2-non-strong-rev-GL(3,H)} imply that the matrix $g$ cannot be an involution. Therefore, there does not exist any $g \in \mathrm{SL}(3,\mathbb{H})$ such that $gAg^{-1} = A^{-1}$ and $g^2 = \mathrm{I}_3$. Hence, $A$ is not strongly reversible in $\mathrm{SL}(3,\mathbb{H})$. This completes the proof.
\end{proof}

The following theorem classifies strongly reversible elements in $\mathrm{SL}(3,\mathbb{H})$.
\begin{theorem}\label{thm-classification-str-rev-SL}
	Let $A \in \mathrm{SL}(3,\mathbb{H})$ be a reversible element.  Then $A$ is strongly reversible if and only if it is conjugate to one of the matrices given in \lemref{lem-strong-rev-GL(3,H)}.
\end{theorem}

\begin{proof}
The proof of theorem follows from \thmref{thm-rev-GL(3,H)}, \lemref{lem-strong-rev-GL(3,H)}  and \lemref{lem-non-strong-rev-GL(3,H)}.
\end{proof}

\section{Reversibility in the projective linear group $\mathrm{PSL}(3,\H)$}\label{sec-rev-projective-group}

In this section, we will investigate the reversibility problem in the projective linear group $\mathrm{PSL}(3,\H)$. Recall that $\mathrm{PSL}(3,\H) := \mathrm{SL}(3,\H)/\{\pm \mathrm{I}_{3}\}$, and for every element $[g] \in \mathrm{PSL}(3,\H)$, there are exactly two lifts, $g$ and $-g$, in the group $\mathrm{SL}(3,\H)$.  Note that every skew-involution  in $\mathrm{SL}(3,\H)$ is an  involution in $\mathrm{PSL}(3,\H)$. Moreover,  for every $g,h \in \mathrm{SL}(3,\H)$,  we have
$$ [g]=[h]  \Leftrightarrow  g=\pm h.$$
The next result relates reversibility in the groups $\mathrm{PSL}(3,\H)$ and $\mathrm{SL}(3,\H) $. 

\begin{lemma}\label{lem-rev-PSL-equiv-conditions}
	An element $[g] \in \mathrm{PSL}(3,\H)$ is reversible if and only if there exists $h \in \mathrm{SL}(3,\H)$ such that either of the following conditions holds.
	\begin{enumerate}
		\item $hgh^{-1}= g^{-1}$.
		\item $hgh^{-1}= -g^{-1}$.
	\end{enumerate}
\end{lemma}
\begin{proof}
	We omit the proof since it is straightforward.
\end{proof}

In the rest of this section, we will investigate the equation $h g  h^{-1}= -g^{-1}$, where $g,h \in \mathrm{SL}(3,\H) $.

\begin{lemma}\label{lem-rev-PSL-type-2}
	An element $A \in \mathrm{SL}(3,\mathbb{H})$  satisfies the equation  $g A  g^{-1}= -A^{-1}$ for some $g \in \mathrm{SL}(3,\H) $,  if and only if it is conjugate in $\mathrm{SL}(3,\mathbb{H})$ to one of the following matrices:
	\begin{enumerate}[(i)]
		\item\label{rev-type-1-PSL} $\begin{pmatrix}
			e^{\ib \theta} &0 &0 \\
			0 &  -e^{-\ib \theta}& 0\\
			0 & 0 & \ib
		\end{pmatrix}$,  where $\theta \in [0,\pi]$,
		
		\vspace{3mm}
		\item \label{rev-type-2-PSL}$\begin{pmatrix}
			re^{\ib \theta} &0 &0 \\
			0 &  - r^{-1}e^{-\ib \theta}& 0\\
			0 & 0 & \ib
		\end{pmatrix}$,  where $r \in \R^{+}, r \neq 1$,  and $\theta \in [0,\pi]$,
		\vspace{3mm}
		\item \label{rev-type-3-PSL} $\begin{pmatrix}
			\ib  &1 &0 \\
			0 &  \ib  & 0\\
			0 & 0 & \ib
		\end{pmatrix},$
		
		\vspace{3mm}
		
		\item  \label{rev-type-4-PSL}$\begin{pmatrix}
			\ib  &1 &0 \\
			0 &  \ib  & 1\\
			0 & 0 & \ib
		\end{pmatrix}$.

	\end{enumerate}

\end{lemma}

\begin{proof}
	Let $g  A  g^{-1} = -A ^{-1}$  for some $g \in \mathrm{SL}(3,\mathbb{H})$.  In view of the Jordan decomposition over $\H$, $A$ is conjugate to $-A^{-1}$ if and only if $A$ and $-A^{-1} $ have same Jordan form.  Now recall  that for a unique complex representative $\lambda$ of an eigenvalue class of  $A$,  $[\lambda]= [- \lambda^{-1}]$ if and only if $\lambda = \pm \ib$.  This implies that the blocks  in the Jordan form of $A$ can be partitioned into pairs $ \{ \mathrm{J}(\alpha, s),\mathrm{J}(-\alpha^{-1}, s)\} $,  or singletons $\{\mathrm{J}(\ib, t  )\}$,  where $\alpha \in \C \setminus \{0\}$  with non-negative imaginary part such that  $\alpha  \neq \ib$.  Hence, $A$ will be conjugate to one of the matrices given in  $(\ref{rev-type-1-PSL})-(\ref{rev-type-4-PSL})$. 
	
	Conversely,  for the matrices given in $(\ref{rev-type-1-PSL})-(\ref{rev-type-4-PSL})$,  we  choose an element  $g \in \mathrm{GL}(3,\mathbb{H})$ in the following way, respectively:
	\begin{enumerate}[(a)]
		\item\label{rev-type-1-inv}  $ g:= \begin{pmatrix}
			0 & 1 &0 \\
			1 &  0 & 0\\
			0 & 0 & 1
		\end{pmatrix}, $
		\vspace{3mm}
		\item \label{rev-type-2-inv}  $g:= \begin{pmatrix}
			0 & 1 &0 \\
			1 &  0 & 0\\
			0 & 0 & 1 
		\end{pmatrix}, $
		\vspace{3mm}
		\item \label{rev-type-3-inv} $g:=    \begin{pmatrix}
			1 &0 &0 \\
			0 &  -1  & 0\\
			0 & 0 & 1 \\
		\end{pmatrix}, $
		\vspace{3mm}
		\item \label{rev-type-4-inv} $g:=  \begin{pmatrix}
			1 & -\ib &0 \\
			0 &  - 1  & 0\\
			0 & 0 & 1 \\
		\end{pmatrix}.$ 
	\end{enumerate}
	Then  $g A g^{-1} = -A ^{-1}$.  This proves the lemma.
\end{proof}

Observe that for each matrix given in \lemref{lem-rev-PSL-type-2}, the conjugating element $g$ constructed in the proof of \lemref{lem-rev-PSL-type-2} is an involution in $\mathrm{SL}(3,\H) $. Therefore,  we have the following result.

\begin{proposition}\label{prop-rev-type-2-PSL-skew-inv-prod}
	Let $A \in \mathrm{SL}(3,\mathbb{H})$  be an element such that $g A g^{-1}= -A^{-1}$ for some   $g \in \mathrm{SL}(3,\H) $. Then $A$ can be written as a product of an involution and a skew-involution in $\mathrm{SL}(3,\H) $.
\end{proposition}
\begin{proof}
	In view of the \lemref{lem-rev-PSL-type-2}, without loss of generality, we can assume that there exists an involution $g \in \mathrm{SL}(3,\H) $ such that $gA g^{-1}= -A^{-1}$. Therefore, 
	$$A = (- g^{-1}  A^{-1})   (g),  \hbox{ and } g^{-1}= g.$$
	This implies  $(- g^{-1}  A^{-1})^2  = (g^{-1}  A^{-1}g^{-1}) (A^{-1}) =  (g^{-1}  A^{-1}  g)  (A^{-1})= - A  A^{-1} = -\mathrm{I}_3$. This completes the proof.
\end{proof}

Unlike $\mathrm{SL}(3,\H)$,  reversible and strongly reversible elements are equivalent in $\mathrm{PSL}(3,\H) $. This is one of our main results from this section and can be proved as follows.

\subsection{Proof of \thmref{thm-main-equiv-PSL-1}}
Note that every involution and skew-involution in $\mathrm{SL}(3,\mathbb{H})$ is an involution in $\mathrm{PSL}(3,\mathbb{H})$. Using Proposition \ref{prop-rev-prod-skew-inv}, Lemma \ref{lem-rev-PSL-equiv-conditions}, and Proposition \ref{prop-rev-type-2-PSL-skew-inv-prod}, we can conclude that every reversible element of $\mathrm{PSL}(3,\mathbb{H})$ can be written as the product of two involutions in $\mathrm{PSL}(3,\mathbb{H})$. Hence, the proof follows.	\qed

\section{Decomposition into simple transformations}\label{sec-decom}
In this section, we will prove  \thmref{th:simple4}. Suppose $\theta$, $\phi$, and $\psi$ are all in the interval $[0,\pi]$. In   Lemmas \ref{lem:se}, \ref{lem:sep}, \ref{lem:set},  assume that $e^{\ib \theta},e^{\ib\phi}$ and $e^{\ib\psi}$ are not all equal to $\pm 1$;  otherwise, the matrix is real and hence simple.

\begin{lemma}
	\label{lem:se}
	If the Jordan form of a matrix $A\in \mathrm{SL}(3,\mathbb{H})$ is
	$\begin{pmatrix}
		e^{\ib \theta} & 0 & 0\\0 & e^{\ib\phi} & 0\\0 & 0 & e^{\ib \psi}
	\end{pmatrix}$ then it can be written as a product of $3$ simple matrices. 
\end{lemma}
\begin{proof}
	Write $A=\begin{pmatrix}
		e^{\ib \theta} & 0 & 0\\0 & e^{\ib\phi} & 0\\0 & 0 & e^{\ib \psi}
	\end{pmatrix}=\begin{pmatrix}
		e^{\ib \theta} & 0 & 0\\0 & e^{\ib\phi'} & 0\\0 & 0 & 1
	\end{pmatrix}\begin{pmatrix}
		1 & 0 & 0\\0 & e^{-\ib \psi} & 0\\0 & 0 & e^{\ib \psi}
	\end{pmatrix},$ where $\phi'=\phi+ \psi$. 
	By considering a suitable embedding  of  $ \mathrm{SL}(2, \mathbb{H})$ into $ \mathrm{SL}(3, \mathbb{H})$
	and then using  \cite[Theorem 1.4]{PS},
	we can show that the first matrix on the right-hand side can be written as a product of $2$ simple matrices, and the second matrix is a simple matrix.   Therefore,  matrix  $A$ can be written as a product of $3$ simple matrices. 
\end{proof}

\begin{lemma}
	\label{lem:sep}
	If the Jordan form of a matrix $A\in \mathrm{SL}(3,\mathbb{H})$ is 
	$\begin{pmatrix}
		e^{\ib \theta} & 1 & 0\\0 & e^{\ib \theta} & 0\\0 & 0 & e^{\ib \psi}
	\end{pmatrix}$ then it can be written as a product of $3$ simple matrices.
\end{lemma}

\begin{proof}
	In this case, $A$ can be written as 
	\begin{align*}
		A=\begin{pmatrix}
			e^{\ib \theta} & 1 & 0\\0 & e^{\ib \theta} & 0\\0 & 0 & e^{\ib \psi}
		\end{pmatrix}
		=\begin{pmatrix}
			e^{-\ib \theta} & 0 & 0\\0 & e^{\ib \theta}& 0\\0 & 0 & 1
		\end{pmatrix}\begin{pmatrix}
			e^{\ib \theta} & e^{2\ib\theta} & 0\\0 & e^{\ib \theta} & 0\\0 & 0 & e^{\ib \psi}
		\end{pmatrix}\begin{pmatrix}
			e^{\ib \theta} & 0 & 0\\0 & e^{-\ib \theta} & 0\\0 & 0 & 1
		\end{pmatrix} =WPW^{-1}.
	\end{align*}
	Further, note that
	\begin{align*}
		P=\begin{pmatrix}
			e^{\ib \theta} & e^{2\ib\theta} & 0\\0 & e^{\ib \theta} & 0\\0 & 0 & e^{\ib \psi}
		\end{pmatrix}
		=\begin{pmatrix}
			e^{\ib \theta} & 1 & 0\\0 & e^{-\ib \theta} & 0\\0 & 0 & 1
		\end{pmatrix}\begin{pmatrix}
			1 & 0 & 0\\0 & e^{2\ib\theta} & 0\\0 & 0 & e^{\ib \psi}
		\end{pmatrix}=QR.
	\end{align*}
By considering a suitable embedding of $\mathrm{SL}(2, \mathbb{H})$ into $\mathrm{SL}(3, \mathbb{H})$ and then using \cite[Theorem 1.4]{PS}, we can show that the matrix $Q$ is a simple matrix, and the matrix $R$ can be written as a product of $2$ simple matrices. Therefore, the matrix $P$ and hence $A$ can be written as a product of $3$ simple matrices.
\end{proof}

\begin{lemma}
	\label{lem:set}
	If the Jordan form of a matrix $A\in \mathrm{SL}(3,\mathbb{H})$ is
	$\begin{pmatrix}
		e^{\ib \theta} & 1 & 0\\0 & e^{\ib \theta} & 1\\0 & 0 & e^{\ib \theta}
	\end{pmatrix}$ then it can be written as a product of $4$ simple matrices. 
\end{lemma}

\begin{proof}
	In this case,  $A$ can be written as
	$$A=\begin{pmatrix}
		e^{\ib \theta} & 0 & 0\\0 & e^{\ib \theta} & 0\\0 & 0 & 1
	\end{pmatrix}\begin{pmatrix}
		1 & 0 & 0\\0 & e^{\ib \theta /2} & 0 \\0 & 0 & e^{\ib \theta /2}
	\end{pmatrix} \begin{pmatrix}
		1 & 0 & 0\\0 & e^{-\ib \theta /2} & 0 \\0 & 0 & e^{\ib \theta /2}
	\end{pmatrix} \begin{pmatrix}
		1 & e^{-\ib \theta} & 0\\0 & 1 & e^{-\ib \theta} & \\0 & 0 & 1
	\end{pmatrix}.$$ 
Since $\jb e^{-\ib \theta} \jb^{-1} = e^{\ib \theta}$, we can show   that  the diagonal matrix $\mathrm{diag}(1,e^{-\ib \theta/2}, e^{\ib \theta/2})$ is conjugate to $\mathrm{diag}(1,e^{\ib \theta/2}, e^{\ib \theta/2})$ in $\mathrm{SL}(3,\mathbb{H})$; see \remref{rem-simple-complex-eigenvalues}. Now, in view of \propref{prop-simple-classification} and \corref{cor-simple-uppertriangular-real-eigenvalue}, we can conclude that all four matrices on the right-hand side of the  above equation are simple.
This proves the lemma.
\end{proof}

Suppose that  $\lambda = |\lambda| 	e^{\ib \theta}$, $\mu = |\mu| e^{\ib\phi}$, and $\xi = |\xi | e^{\ib\psi}$  are  complex numbers with non-negative imaginary parts.
In the following, for  Lemmas \ref{lem:sd} and \ref{lem:slp}, assume that at least one of the $\lambda,\mu$ or, $\xi$ are not of unit modulus; otherwise, refer to Lemma \ref{lem:se}, \ref{lem:sep} and \ref{lem:set}. In Lemmas \ref{lem:sd} and \ref{lem:slp}, at least one of the $\lambda,\mu$ and $\xi$ are not in $\R$; otherwise, the matrix is real and hence simple.

\begin{lemma}
	\label{lem:sd}
	If the Jordan form of a matrix $A\in \mathrm{SL}(3,\mathbb{H})$ is
	$\begin{pmatrix}
		\lambda & 0 & 0\\0 & \mu & 0\\0 & 0 & \xi	\end{pmatrix}$ then it can be written as a product of $4$ simple matrices. 
\end{lemma}
\begin{proof} In this case, $A$ can be written as 
	\begin{align*}
		A=	\begin{pmatrix}
			\lambda & 0 & 0\\0 & \mu & 0\\0 & 0 & \xi	\end{pmatrix}
		&=\begin{pmatrix}
			|\lambda| & 0 & 0\\0 & |\mu| & 0\\0 & 0 & |\xi|	\end{pmatrix}\begin{pmatrix}
			e^{\ib \theta} & 0 & 0\\0 & e^{\ib\phi} & 0\\0 & 0 & e^{\ib \psi} 	\end{pmatrix} = PQ.
	\end{align*}
	Note that $P$ is a real matrix and hence a simple matrix in $\mathrm{SL}(3,\mathbb{H})$. Using \lemref{lem:se} we can write $Q$ as a product of $3$ simple matrices. This proves the lemma.
\end{proof}

\begin{lemma}
	\label{lem:slp}
	If the Jordan form of a matrix $A\in \mathrm{SL}(3,\mathbb{H})$ is
	$\begin{pmatrix}
		\lambda & 1 & 0\\0 & \lambda & 0\\0 & 0 & \xi
	\end{pmatrix}$ then it can be written as a product of $4$ simple matrices. 
\end{lemma}
\begin{proof} In this case, $A$ can be written as 
	$$A=\begin{pmatrix}
		\lambda & 1 & 0\\0 & \lambda & 0\\0 & 0 & \xi
	\end{pmatrix}=\begin{pmatrix}
		|\lambda| & e^{-\ib\theta} & 0\\0 & |\lambda| & 0\\0 & 0 & |\xi|
	\end{pmatrix}\begin{pmatrix}
		e^{\ib \theta} & 0 & 0\\0 & e^{\ib \theta} & 0\\0 & 0 & e^{\ib \psi}
	\end{pmatrix}= PQ.$$ 
	The matrix $P$  is conjugate to a real matrix $\begin{pmatrix}
		|\lambda| & 1 & 0\\0 & |\lambda| & 0\\0 & 0 & |\xi|
	\end{pmatrix}$ and hence $P$ is a simple matrix; see  \corref{cor-simple-uppertriangular-real-eigenvalue}.
	Using \lemref{lem:se} we can write $Q$ as a product of $3$ simple matrices. This proves the lemma. 
\end{proof}

Now the  Table \ref{table:1} follows from   Lemmas \ref{lem:se},   \ref{lem:sep},  \ref{lem:set},  \ref{lem:sd}, \ref{lem:slp}. We will use the above results to prove  Theorem \ref{th:simple4}.
\subsection{Proof of \thmref{th:simple4}}
Let $A \in \mathrm{SL}(3,\mathbb{H})$. Then, using the Jordan decomposition, $A$ is conjugate in $\mathrm{SL}(3,\mathbb{H})$ to one of the Jordan forms considered in this section. Hence, in view of Lemmas \ref{lem:se}, \ref{lem:sep}, \ref{lem:set}, \ref{lem:sd}, and \ref{lem:slp}, we can write $A$ as a product of $4$ simple matrices. This completes the proof. \qed

\section{Classification in $\mathrm{SL}(3,\mathbb{R})$} \label{sec:sl3r}
In this section, we will prove Theorem \ref{th:class1} by examining the subgroup of $\mathrm{SL}(3,\mathbb{H})$ consisting of transformations with real entries and determinant one. Let $A \in \mathrm{SL}(3,\mathbb{R}) \subset \mathrm{SL}(3,\mathbb{H})$. Since ${\rm det_{\H}(A)}={\rm det(\Phi(A))}$, it follows that ${\rm det_{\H}(A)}=(\mathrm{det}(A))^2=1$, and thus $\mathrm{det}(A)=\pm 1$. We will identify $\mathrm{SL}(3,\mathbb{R})$ with the component of transformations with real entries and $\mathrm{det}(A)=1$.

Let $A\in  \mathrm{SL}(3,\mathbb{R})$. Then  the characteristic polynomial of $A$ is $\chi_A(t)=t^3-xt^2+yt-1$, where $x=\text{tr}(A)$ and  $y=\text{tr}(A^{-1})$.   If $\lambda_1,\lambda_2,\lambda_3$ are the eigenvalues of $A$, then we have  $x=\lambda_1+\lambda_2+\lambda_3$,  $y=\lambda_1^{-1}+\lambda_2^{-1}+\lambda_3^{-1}$, and  $\mathrm{det}(A)= \lambda_1 \lambda_2 \lambda_3 =1$. In the next lemma, we will provide a characterization of the eigenvalues of $A$ in terms of $x$ and $y$.

\begin{lemma}\label{lem:xy}
	All the eigenvalues of an element $A\in  \mathrm{SL}(3,\mathbb{R})$ are of unit modulus if and only if $x=y$,  where $x=\text{tr}(A)$ and $y=\text{tr}(A^{-1})$.
\end{lemma}

\begin{proof}
	Suppose  all the eigenvalues of $A$ are $\lambda_1,\lambda_2,\lambda_3$ and they are of unit modulus. Since the characteristic polynomial $\chi_A(t)$ has odd degree, it has at least one real root.  Suppose $\lambda_1\in \R$. Then we have $\lambda_1=\pm 1$ and $\lambda_1=\lambda_1^{-1}$. If $\lambda_2$ and $\lambda_3$ are also real  then we have $\lambda_i=\lambda_i^{-1}$ for $i=2,3$. Consequently, we get $x=y$. Assuming $\lambda_{2}=e^{\ib\theta}$, it follows that $\lambda_{3}=e^{-\ib\theta}$. Hence, $\lambda_2+\lambda_3=e^{\ib \theta}+e^{-\ib \theta}= \lambda_2^{-1} + \lambda_3^{-1}$, which implies that $x=y$.
	
	Conversely, let us assume that $x=y$. Suppose the eigenvalues of $A$ are $\lambda_1,\lambda_2,$ and $\lambda_3$, given by $r^{-2},re^{\ib\theta},$ and $re^{-\ib\theta}$, respectively, where $r$ is a positive real number and $\theta\in [0,\pi]$. If $r=1$, we are done.
	However, if $r$ is not equal to $1$, then note that $x=y$ implies $r^{-2}+2r\cos\theta=r^2+2r^{-1}\cos\theta$. By suitably pairing the terms, we obtain $r+r^{-1}=2\cos\theta$ since $r-r^{-1} \neq 0$. Therefore, $r=e^{\pm \ib\theta}$. However, since $r$ is a positive real number, the only possibility is $r=1$. Hence, this contradicts the assumption that $r$ is not equal to $1$. Therefore, $r=1$, and the eigenvalues are $1, e^{\pm \ib\theta}$.
\end{proof}

Let $R(\chi_A,\chi_A')$ be the resultant of polynomials $\chi_A$ and $\chi_A'$, defined as the determinant of the \textit{Sylvester matrix}:
$$ R(\chi_A,\chi_A')=\begin{vmatrix}
	1 & -x & y & -1 & 0\\0 & 1 & -x & y & -1\\3 & -2x & y & 0 & 0\\
	0 & 3 & -2x & y & 0\\0 & 0 & 3 & -2x & y
\end{vmatrix} = -x^2y^2+4(x^3+y^3)-18xy+27.$$ 
Let $\triangle(\chi_A)$ be the  discriminant  of $\chi_A$.  Then we have 
$$\triangle(\chi_A) =  R(\chi_A,\chi_A') = \prod_{1\leq i<j\leq 3} (\xi_i-\xi_j)^2, $$
where $\xi_i$'s are the roots of $\chi_A(t)$; see \cite[Chapter 3]{Co}.
Note that $R(\chi_A,\chi_A')=0$ if and only if $\chi_A$ and $\chi_{A}'$ have a common root, i.e., $\chi_A$ has a multiple root. 	We will now prove  \thmref{th:class1}.

\subsection{Proof of Theorem \ref{th:class1}} Here,  we have $ f(x,y) =  R(\chi_A,\chi_A')$ and $\chi_A(t)= t^3-xt^2+yt-1$. Note that $\chi_A$ is a degree $3$ monic polynomial with real coefficients. Therefore, if $re^{\ib \theta}$ is a root of $\chi_A$  then  $re^{-\ib \theta}$ is also  a root of   $\chi_A$. 
Since  $f(x,y)=  R(\chi_A,\chi_A')= \prod_{1\leq i<j\leq 3} (\xi_i-\xi_j)^2, $ where $\xi_i$'s are roots of $\chi_A(t)$, it is easy to observe that the following statements are true.
\begin{enumerate}
	\item [(\textit{i})] $f(x,y)>0$ if and only if  $\chi_A$  has $3$ distinct real  roots.
	\item  [(\textit{ii})]  $f(x,y)<0$ if and only if  $\chi_A$  has $1$ real  root and $2$ conjugate complex roots.
	\item [(\textit{iii})] $f(x,y)=0$ if and only if  $\chi_A$ has a multiple root.
\end{enumerate}

Suppose $A \in \mathrm{SL}(3,\mathbb{R})$ is a regular loxodromic. Then, all the eigenvalues of $A$ are real numbers with distinct moduli. It follows that $f(x,y)>0$.
Conversely, suppose $f(x,y)>0$. Then, all roots of $\chi_A$ are distinct real numbers. Therefore, $A$ is diagonalizable. Since all eigenvalues have distinct moduli, $A$ is a regular loxodromic.
Hence, we conclude that $A$ is a regular loxodromic if and only if $f(x,y)>0$.

Assuming $f(x,y)<0$, we can conclude that $\chi_A$ has $1$ real root and $2$ conjugate non-real roots. By applying Lemma \ref{lem:xy}, it follows that $x=y$ if and only if all the eigenvalues are of unit modulus. Therefore, we can conclude that $A$ is a regular elliptic if and only if $f(x,y)<0$, $x=y$.
Similarly, we can deduce  that $A$ is a screw loxodromic if and only if $f(x,y)<0$, $x\neq y$.

We can consider the following cases for $f(x,y)=0$.
\begin{enumerate}
	\item If
	$f(x,y)=0$ and $x \neq y$,  then $\chi_A$ has a multiple root and at least one root of non-unit modulus. In this case,  $A$ is a homothety or a loxo-parabolic, depending on whether it is diagonalizable or not.
	\item  If $f(x,y)=0$ and $x=y$,   then $\chi_A$ has a multiple root, and all roots have unit modulus. Thus  eigenvalues of $A$ are either $1$ or $-1$.  There are a total of four such non-trivial elements up to conjugacy, and they are given by:
	\begin{equation*}
		\begin{pmatrix}
			-1 & 0 & 0\\0 & -1 & 0\\0 & 0 & 1
		\end{pmatrix}, \; \begin{pmatrix}
			1 & 1 & 0\\0 & 1 & 0\\0 & 0 & 1
		\end{pmatrix},\;\begin{pmatrix}
			1 & 1 & 0\\0 & 1 & 1\\0 & 0 & 1
		\end{pmatrix}\text{ and }\begin{pmatrix}
			-1 & 1 & 0\\0 & -1 & 0\\0 & 0 & 1
		\end{pmatrix}.
	\end{equation*}
	Out of the four matrices mentioned above, the last three matrices are not diagonalizable. The first matrix is diagonal and is classified as an elliptic-reflection (or reflection). The second and third matrices are categorized as unipotent matrices and have minimal polynomials equal to $(x-1)^2$ and $(x-1)^3$, respectively. The second and third matrices are also known as vertical translation and non-vertical translation, respectively; see Section \ref{sec-prel}. Finally, the last one is an ellipto-parabolic and has a minimal polynomial equal to $(x-1)(x+1)^2$.
\end{enumerate} 
This completes the proof. \qed

\section{Concluding remarks}\label{sec-final-class}
In this section, we propose an algorithm for classifying the transformations in $\mathrm{SL}(3,\mathbb{H})$.  The classification of a general element in $\mathrm{SL}(3,\H)$  can be seen as follows:
\begin{enumerate}
	\item \textit{Reversible elements:} Let $A \in \mathrm{SL}(3,\H)$ be a reversible element. Then the characteristic polynomial $\chi_{\Phi (A)}$ of $A$ is self-dual, i.e., it can be expressed as $\chi_{\Phi (A)}=x^6-a x^5 + b x^4 -c x^3+b x^2 -a x +1$, where $a,b,c \in \R$.
	Therefore,  we can use the approach of \cite{KG}. Hence, the classification of reversible elements of $\mathrm{SL}(3,\H)$ in terms of coefficients of their characteristic polynomial will follow. 
	\item \textit{Non-reversible elements:}   In view of the \thmref{thm-rev-GL(3,H)}, it follows that all the non-reversible elements of $\mathrm{SL}(3,\H)$ are loxodromic,  and up to conjugacy,  they are listed in \remref{remark_non-rev-SL(3,H)}. 
\end{enumerate}

We classify the simple elements of $ \mathrm{SL}(3,\mathbb{H})$ by the following algorithm:
\begin{enumerate}
	\item  [(\textit{i})]First, check whether the matrix $A\in \mathrm{SL}(3,\mathbb{H})$ is simple or not by using the \propref{prop-simple-classification}.
	\item [(\textit{ii})] If $A$ is simple, then find a matrix $B\in  \mathrm{SL}(3,\mathbb{R})$ which is  conjugate to $A$.
	\item [(\textit{iii})] Find the classification of $B$ by Theorem \ref{th:class1}.
	\item  [(\textit{iv})] Finally, conclude that $A$ is elliptic, parabolic, or loxodromic according to the classification of $B$.
\end{enumerate}

\textbf{Acknowledgement.} 
The authors would like to thank J. Kim for his comments on the first draft of this paper. It is a great pleasure to thank the referee for carefully reading the paper and for providing many valuable comments.
	
Gongopadhyay is partially supported by the SERB core research grant CRG/2022/003680. Lohan acknowledges full support from the CSIR SRF grant,  file no.:  09/947(0113)/2019-EMR-I, during the course of this work.

\end{document}